\documentclass[graybox]{svmult}

\usepackage{mathptmx}       
\usepackage{helvet}         
\usepackage{courier}        
\usepackage{type1cm}        
\usepackage{makeidx}         
\usepackage{graphicx}        
\usepackage{multicol}        
\usepackage[bottom]{footmisc}

\usepackage{amssymb}
\usepackage{amsmath}

\usepackage{url}

\usepackage{caption}
\spdefaulttheorem{project}{Research Project}{\bf}{}
\newcommand{\resproject}[1]{\begin{svgraybox} \begin{project} #1 \end{project} \end{svgraybox}}

\spdefaulttheorem{cproblem}{Challenge Problem}{\bf}{}

\spnewtheorem*{prerequisites}{Suggested prerequisites}{\bf}{\small\em}

\makeindex             

\newcommand{\excise}[1]{}

\renewcommand\>{\rangle}
\newcommand\<{\langle}

\newcommand{\ZZ}{\mathbb{Z}}

\begin{document}

\title*{Beyond Coins, Stamps, and Chicken McNuggets:\ an Invitation to Numerical Semigroups}
\titlerunning{An Invitation to Numerical Semigroups}
\author{Scott Chapman, Rebecca Garcia, and Christopher O'Neill}
\institute{Scott Chapman \at Mathematics and Statistics Department, Sam Houston State University, Huntsville, TX 77340, \email{scott.chapman@shsu.edu}
\and Rebecca Garcia \at Mathematics and Statistics Department, Sam Houston State University, Huntsville, TX 77340, \email{mth_reg@shsu.edu}
\and Christopher O'Neill \at Mathematics and Statistics Department, San Diego State University, San Diego, CA 92182 \email{cdoneill@sdsu.edu}}
%
%

\maketitle

\abstract{
We give a self contained introduction to numerical semigroups, and present several open problems centered on their factorization properties.
}

\begin{prerequisites}
Linear Algebra, Number Theory.
\end{prerequisites}


\section{Introduction}
\label{sec:intro}

Many difficult mathematics problems have extremely simple roots.  For instance, suppose you walk into your local convenience store to buy that candy
bar you more than likely should not eat.  Suppose the candy bar costs $X$ cents and you have $c_1$ pennies, $c_2$ nickels, $c_3$ dimes, and $c_4$ quarters 
in your pocket (half dollars are of course too big to carry in your pocket).  Can you buy the candy bar?  You can if there are non-negative integers $x_1, \ldots, x_4$ such
that 
\[
x_1 + 5x_2 + 10x_3 +25x_4 \geq X
\]
with $0\leq x_i\leq c_i$ for each $1\leq i\leq 4$.  Obviously, this is not difficult mathematics; it is a calculation that almost everyone goes through in their heads multiple times a week.  Now, suppose the cashier indicates that the register is broken, and that the store can only accept the exact amount of money necessary in payment for the candy bar.  This changes the problem to 
\[
x_1 + 5x_2 + 10x_3 +25x_4 = X
\]
with the same restrictions on the $x_i$'s.  

Given our change system, the two equations above are relatively easy with which to deal, but changing the values of the coins involved can make the problem much more difficult.  For instance, instead of our usual change, suppose you have a large supply of 3-cent pieces and 7-cent pieces.  Can you buy the 11-cent candy bar?  With a relatively gentle calculation, even your English major roommate concludes that you cannot.  There is no solution of $3x_1 + 7x_2=11$ in the non-negative integers.  But with a little more tinkering, you can unearth a deeper truth.
\medskip

\noindent\textbf{Big Fact:} \textit{In a 3-7 coin system, you can buy any candy bar costing above 11 cents.}
\medskip

The Big Fact follows since $12 = 3 \cdot 4$, $13 = 1 \cdot 7 + 2 \cdot 3$, $14 = 2 \cdot 7$, and any integer value greater than 14 can be obtained by adding the needed number of 3 cent pieces to one of these sums.  But why limit the fun to coins?  Analogous problems can be constructed using postage stamps and even Chicken McNuggets.  The key to what we are doing involves an interesting mix of linear algebra\index{linear algebra}, number theory\index{number theory}, and abstract algebra\index{abstract algebra}, and quickly leads to some simply stated mathematics problems that are very deep and remain (over a long period of time) unsolved.  Moreover, these problems have been the basis of a wealth of undergraduate research projects, many of which led to publication in major mathematics research journals.  In order to discuss these research level problems, we will now embark on a more technical description of the work at hand.  As our pages unfold, the reader should keep in mind the humble beginnings of what will become highly challenging work. 


A \emph{numerical semigroup}\index{numerical~semigroup} is a subset $S \subset \ZZ_{\ge 0}$ of the non-negative integers that 
\medskip

(i) is closed under addition, i.e., whenever $a, b \in S$, we also have $a + b \in S$, and \\
\indent (ii) has finite complement in $\ZZ_{\ge 0}$. \\

\noindent The two smallest examples of numerical semigroups are $S = \ZZ_{\ge 0}$ and $S = \ZZ_{\ge 0} \setminus \{1\}$.  Often, the easiest way to specify a numerical semigroup is by providing a list of \emph{generators}\index{numerical~semigroup!generators}.  For instance,
\[
\<n_1, \ldots, n_k\> = \{a_1n_1 + \cdots + a_kn_k : a_1, \ldots, a_k \in \ZZ_{\ge 0}\}
\]
equals the set of all non-negative integers obtained by adding copies of $n_1, \ldots, n_k$ together.  The smallest nontrivial numerical semigroup $S = \ZZ_{\ge 0} \setminus \{1\}$ can then also be written as $S = \<2,3\>$, since every non-negative even integer can be written as $2k$ for some $k \ge 0$, and every odd integer greater than 1 can be written as $2k + 3$ for some $k \ge 0$.  It is clear that generating systems are not unique (for instance, $\<2,3\>=\<2,3,4\>$), but we will argue later than each numerical semigroup has a unique generating set\index{numerical~semigroup!generating~set} of minimal cardinality.  Note that problems involving the 3-7 coin system take place in the numerical semigroup $\<3,7\>$. 

\begin{example}
Although numerical semigroups may seem opaque, there are some very practical ways to think about them.  For many years, McDonald's sold Chicken McNuggets in packs of 6, 9, and 20, and as such, it is possible to buy exactly $n$ Chicken McNuggets using only those three pack sizes precisely when $n \in \<6, 9, 20\>$.  For this reason, the numerical semigroup $S = \<6,9,20\>$ is known as the \emph{McNugget semigroup}\index{McNugget~semigroup} (see \cite{COn18}).  It turns out that it is impossible to buy exactly 43 Chicken McNuggets using only packs of 6, 9, and 20, but for any integer $n > 43$, there is some combination of packs that together contain exactly $n$ Chicken McNuggets.  
\end{example}

By changing the quantities involved (be it with coins or Chicken McNuggets) yields what is known in the literature as the  
\emph{Frobenius coin-exchange problem}\index{Frobenius~coin-exchange~problem}.  To Frobenius, each generator of a numerical semigroup corresponds to a coin denomination, and the largest monetary value for which one cannot make even change is the \emph{Frobenius number}\index{Frobenius~number}.  In terms of numerical semigroups, the \emph{Frobenius number} of $S$ is given by 
\[
F(S) = \max(\ZZ_{\ge 0} \setminus S).
\]
Sylvester proved in 1882 (see \cite{Syl1882}) that in the 2-coin problem (i.e., if $S = \<a, b\>$ with $\gcd\,(a,b)=1$), the Frobenius number is given by $F(S) = ab - (a + b)$.  To date, a general formula for the Frobenius number of an arbitrary numerical semigroup (or even a ``fast'' algorithm to compute it from a list of generators) remains out of reach.  

While deep new results concerning the Frobenius number are likely beyond the scope of a reasonable undergraduate research project, a wealth of problems related to numerical semigroups have been a popular topic in REU programs for almost 20 years.  To better describe this work, we will need some definitions.  Assume that $n_1, \ldots, n_k$ is a set of generators for a numerical semigroup $S$.  For $n \in S$, we refer to
\[
\mathsf Z(n) = \bigg\{(x_1, \ldots, x_k) \, \mid \, n = \sum_{i=1}^k x_in_i\bigg\}
\]
as the \emph{set of factorizations}\index{set~of~factorizations} of $n \in S$, and to
\[
\mathsf L(n) = \bigg\{\sum_{i=1}^k x_k \, \mid \, (x_1, \ldots, x_k) \in \mathsf Z(n)\bigg\}
\]
as the \emph{set of factorization lengths}\index{set~of~factorization~lengths} of $n \in S$.  Each element of $\mathsf Z(n)$ represents a distinct \emph{factorization} of $n$ (that is, an expression 
$$n = x_1n_1 + \cdots + x_kn_k$$
of $n$ as a sum of $n_1, \ldots, n_k$, wherein each $x_i$ denotes the number of copies of $n_i$ used in the expression).  The local descriptors $\mathsf Z(n)$ and $\mathsf L(n)$ can be converted into global descriptors of $S$ by setting
\[
\mathcal Z(S) = \{\mathsf Z(n) \, \mid \, n \in S\}
\]
to be the \emph{complete set of factorizations}\index{complete~set~of~factorizations} of $S$ and 
\[
\mathcal L(S) = \{\mathsf L(n) \, \mid \, n \in S\}
\]
to be the \emph{complete set of lengths}\index{complete~set~of~lengths} of $S$ (note that these are both sets of sets).  

Hence, while we started by exploring the membership problem for a numerical semigroup (i.e., given $m \in \ZZ_{\ge 0}$, is $m \in S$?), we now focus on two different questions.

\begin{enumerate}
\item 
Given $n \in S$ what can we say about the set $\mathsf Z(n)$?

\item 
Given $n \in S$, what can we say about the set $\mathsf L(n)$?

\end{enumerate}

\noindent We start with a straightforward but important observation.

\begin{exercise}
If $S$ is a numerical semigroup and $n \in S$, then $\mathsf Z(n)$ and $\mathsf L(n)$ are both finite sets.
\end{exercise}

\begin{example}\label{e:2gen}
Calculations of the above sets tend to be nontrivial and normally require some form of a computer algebra system.  To demonstrate this, we return to the elementary example $S = \<2,3\>$ mentioned earlier.  As previously noted, any integer $n \geq 2$ is in $S$.  In Table~\ref{tb:2gen}, we give $\mathsf Z(n)$ and $\mathsf L(n)$ for some basic values of $n \in S$.  
\end{example}

\begin{table}[t]
\[
\begin{array}{||c|c|c||c|c|c||}
\hline
n & \mathsf Z(n) & \mathsf L(n) & n & \mathsf Z(n) & \mathsf L(n) \\
\hline
\mathbf{2} & \{(1,0)\} & \{1\} &
\mathbf{11} & \{(1,3),(4,1)\} & \{4,5\} \\
\hline
\mathbf{3} & \{(0,1)\} & \{1\} &
\mathbf{12} & \{(6,0),(3,2),(0,4)\} & \{4,5,6\} \\
\hline
\mathbf{4} & \{(2,0)\} & \{2\} & 
\mathbf{13} & \{(5,1),(2,3)\} & \{5,6\} \\
\hline
\mathbf{5} & \{(1,1)\} & \{2\} &
\mathbf{14} & \{(7,0), (4,2), (1,4)\} & \{5,6,7\} \\
\hline
\mathbf{6} & \{(3,0), (0,2)\} & \{2,3\} &
\mathbf{15} & \{(6,1),(3,3),(0,5)\} & \{5,6,7\} \\
\hline
\mathbf{7} & \{(2,1)\} & \{3\} &
\mathbf{16} & \{(8,0),(5,2),(2,4)\} & \{6,7,8\} \\
\hline 
\mathbf{8} & \{(4,0),(1,2)\} & \{3,4\} &
\mathbf{17} & \{(7,1),(4,3),(1,5)\} & \{6,7,8\} \\
\hline
\mathbf{9} & \{(3,1),(0,3)\} & \{3,4\} &
\mathbf{18} & \{(9,0), (6,3), (3,4), (0,6)\} & \{6,7,8,9\} \\
\hline
\mathbf{10} & \{(5,0),(2,2)\} & \{4,5\} &
\mathbf{19} & \{(8,1), (5,3), (2,5)\} & \{7,8,9\} \\
\hline 
\end{array}
\]
\caption{Some basic values of $\mathsf Z(n)$ and $\mathsf L(n)$ where $S = \<2,3\>$.}
\label{tb:2gen}
\end{table}

\begin{example}\label{e:3gen}
Patterns in the last example are easy to identify (and we will return to Example~\ref{e:2gen} in our next section), but the reader should not be too complacent, as the two generator case is the simplest possible.  We demonstrate this by producing in Table~\ref{tb:3gen} the same sets, now for the semigroup $S = \<7,10,12\>$.  We make special note that while the length sets in Table~\ref{tb:2gen} are sets of consecutive integers, $\mathsf L(42) = \{4, 6\}$ in Table~\ref{tb:3gen} breaks this pattern.  We will revisit the concept of ``skips'' in length sets at the end of the next section.  
\end{example}

\begin{table}[t]
\[
\begin{array}{||c|c|c||c|c|c||}
\hline
n & \mathsf Z(n) & \mathsf L(n) & n & \mathsf Z(n) & \mathsf L(n) \\
\hline
\mathbf{7} & \{(1,0,0)\} & \{1\} &
\mathbf{30} & \{(0,3,0)\} & \{3\} \\
\hline
\mathbf{10} & \{(0,1,0)\} & \{1\} &
\mathbf{31} & \{(3,1,0),(1,0,2)\} & \{3,4\} \\
\hline
\mathbf{12} & \{(0,0,1)\} & \{1\} & 
\mathbf{32} & \{(0,2,1)\} & \{3\} \\
\hline
\mathbf{14} & \{(2,0,0)\} & \{2\} &
\mathbf{33} & \{(3,0,1)\} & \{4\} \\
\hline
\mathbf{17} & \{(1,1,0)\} & \{2\} &
\mathbf{34} & \{(2,2,0),(0,1,2)\} & \{3,4\} \\
\hline
\mathbf{19} & \{(1,0,1)\} & \{2\} &
\mathbf{35} & \{(5,0,0)\} & \{5\} \\
\hline 
\mathbf{20} & \{(0,2,0)\} & \{2\} &
\mathbf{36} & \{(2,1,1),(0,0,3)\} & \{3,4\} \\
\hline
\mathbf{21} & \{(3,0,0)\} & \{3\} &
\mathbf{37} & \{(1,3,0)\} & \{4\} \\
\hline
\mathbf{22} & \{(0,1,1)\} & \{2\} &
\mathbf{38} & \{(4,1,0), (2,0,2)\} & \{4,5\} \\
\hline 
\mathbf{24} & \{(2,1,0),(0,0,2)\} & \{2,3\} &
\mathbf{39} & \{(1,2,1)\} & \{4\} \\
\hline 
\mathbf{26} & \{(2,0,1)\} & \{3\} &
\mathbf{40} & \{(0,4,0),(4,0,1)\} & \{4,5\} \\
\hline
\mathbf{27} & \{(1,2,0)\} & \{3\} &
\mathbf{41} & \{(3,2,0),(1,1,2)\} & \{4,5\} \\
\hline
\mathbf{28} & \{(4,0,0)\} & \{4\} &
\mathbf{42} & \{(6,0,0), (0,3,1)\} & \{4,6\} \\
\hline
\mathbf{29} & \{(1,1,1)\} & \{3\} &
\mathbf{43} & \{(3,1,1), (1,0,3)\} & \{4,5\} \\
\hline
\end{array}
\]
\caption{Some basic values of $\mathsf Z(n)$ and $\mathsf L(n)$ where $S = \<7,10,12\>$.}
\label{tb:3gen}
\end{table}

Much of the remainder of this paper will focus on the study of $\mathcal Z(S)$ and $\mathcal L(S)$, the complete systems of factorizations and factorizations lengths of $S$.  The next section presents a crash course on definitions and basic results.  We will review some of the significant results in this area, with an emphasis on those obtained in summer and year long REU projects.  Section~\ref{sec:computing} explores the computation tools available to embark on similar studies, and Sections~\ref{sec:asymptotics} and~\ref{sec:random} contain actual student level projects which we hope will peak students' minds and interests.


\section{A crash course on numerical semigroups}
\label{sec:overview}

We start with a momentary return to the notion of minimal generating sets alluded to in Section~\ref{sec:intro}.  Let $S$ be a numerical semigroup, and $m > 0$ its smallest positive element.  We call a generating set $W$ for $S$ \emph{minimal}\index{numerical~semigroup!generating~set!minimal} if $W \subseteq T$ for any other generating set $T$ of $S$.   We claim any minimal generating set for $S$ has at most $m$ elements (and in particular is finite).  Indeed, for each $i$ with $0 \leq i < m$, set
\[
M_i = \{n \in S \,\mid\, n > 0 \mbox{  and  } n \equiv i \bmod m\}.
\]
Note that by the definition of $S$, each $M_i \neq \emptyset$ (in fact, each is infinite).  Moreover, for each $n \in M_i$, we must have $n + m \in M_i$ as well since $S$ is closed under addition.  Hence, setting
\[
n_i = \min M_i
\]
for each $i$, we can write each $M_i = \{n_i + qm \,\mid\, q \ge 0\}$.  This implies $N = \{n_0, \ldots, n_{m-1}\}$ is a generating set for $S$, i.e.,
\[
S = \<N\> = \<n_0, \ldots, n_{m-1}\>,
\]
since the remaining elements of $S$ can each be obtained from an element of $N$ by adding $m = n_0$ sufficiently many times (note that this is precisely the argument used to justify the Big Fact at the beginning of Section~\ref{sec:intro}).  As a consequence, any minimal generating set for $S$ must be a subset of $N$.  

\begin{example}\label{e:mingens}
While the elements of $N$ are chosen with respect to minimality modulo~$m$, the generating set $N$ may not be minimal.   For instance, if 
\[
S = \{0, 5, 8, 10, 13, 15, 16, 18, 20, 21, 23, 24, 25, 26, 28, 29, 30, 31, 32, \ldots\},
\]
then $N = \{ 5, 16, 32, 8, 24\}$, and while $S = \<5, 8, 16, 24, 32\>$, the fact that 
$16 = 2 \cdot 8$, $24 = 3 \cdot 8$, and $32 = 4 \cdot 8$ yields $S = \<5,8\>$.  
\end{example}

Using Example~\ref{e:mingens}, we can reduce $N$ to a minimal generating set as follows.
For each $i$, set $N_i = N \setminus \{n_i\}$, and set
\[
\widehat{N} = \{n_i \,\mid\, n_i \notin \<N_i\>\}.
\]
The following is a good exercise, and implies that every numerical semigroup has a unique minimal generating set.  

\begin{exercise}\label{exe:mingens}
If $S$ is a numerical semigroup and $T$ is any generating set of $S$, then $\widehat{N} \subseteq T$.  In particular, $\widehat{N}$ is the unique minimal generating set of $S$.
\end{exercise}

Exercise~\ref{exe:mingens} and the argument preceding it establish some characteristics of a numerical semigroup $S$ that are widely used in the mathematics literature.  The smallest positive integer in $S$ is called the \emph{multiplicity}\index{numerical~semigroup!multiplicity} of $S$ and denoted by $m(S)$.   The cardinality of $\widehat{N}$ above is called the \emph{embedding dimension}\index{numerical~semigroup!embedding~dimension} of $S$ and is denoted by $e(S)$.  By the argument preceding Exercise~\ref{exe:mingens}, $e(S) \leq m(S)$, and additionally the elements of $\widehat{N}$ are pairwise incongruent modulo $m(S)$.  Moreover, note that $\gcd(\widehat{N}) = 1$, as otherwise $\ZZ_{\ge 0} \setminus S$ would be infinite.

\begin{example}\label{ex5}
Using Exercise~\ref{exe:mingens}, we can set up several obvious classes of numerical semigroups which have garnered research attention.  Suppose that 
$m \ge 2$ and $d \ge 1$ are integers with $\gcd(m,d) = 1$.  Given $k$ with $1 \leq k \leq m-1$, set
\[
A_k = \{m, m + d, \ldots, m + kd\}.
\]
Using elementary number theory, it is easy to see that $A_k$ is the minimal generating set of the numerical semigroup $\<A_k\>$, which is called an \emph{arithmetical numerical semigroup}\index{numerical~semigroup!arithmetical} (since its minimal generating set is an arithmetical sequence).  This is a very large class of numerical semigroups, which contains many important subclasses:
\begin{itemize}
\item
all 2-generated\index{numerical~semigroup!2-generated} numerical semigroups (i.e., $k = 1$);

\item
all numerical semigroups generated by consecutive integers (i.e., $d = 1$); and

\item
numerical semigroups consisting of all positive integers greater than or equal to a fixed positive integer $m$ (i.e., $d = 1$, $k = m - 1$, and $S = \<m, m + 1, \ldots , 2m - 1\>$).

\end{itemize}
The latter subclass consists of all numerical semigroups for which $F(S) < m(S)$.  
\end{example}  

Just as we factor integers as products of primes, or polynomials as products of irreducible factors, we now factor elements in a numerical semigroup $S$ in terms of its minimal generators (in this context, ``factorization'' means an expression of an element of $S$ as a sum of generators, and as we will see, many elements have multiple such expressions).  In terms of $S$, we have already defined the notation $\mathsf Z(n)$, $\mathsf L(n)$, $\mathcal Z(S)$, and $\mathcal L(S)$.  Let's consider some further functions that concretely address structural attributes of these sets.  We~denote the maximum\index{factorization~length!maximum} and minimum\index{factorization~length!minimum} factorization lengths\index{factorization~length} of an element $n \in S$ by 
\[
\ell(n) = \min\mathsf L(n) \quad \mbox{and} \quad L(n) = \max\mathsf L(n).
\] 
These functions satisfy the following recurrence for sufficiently large semigroup elements; we state this result now and revisit it in much more detail in Section~\ref{sec:computing}.  

\begin{theorem}[{\cite[Theorems~4.2 and~4.3]{elastsets}}]\label{t:maxminlen}
If $S = \<n_1, \ldots, n_k\>$ with $n_1 < \cdots < n_k$, then
\[
\ell(n + n_k) = \ell(n) + 1 \quad \mbox{and} \quad L(n + n_1) = L(n) + 1
\]
for all $n > n_{k-1}n_k$.  
\end{theorem}

The \emph{elasticity}\index{elasticity} of a nonzero element $n \in S$, denoted $\rho(n)$, measures the deviation between $\ell(n)$ and $L(n)$ and is defined as 
\[
\rho(n) = L(n)/\ell(n).
\]
The \emph{elasticity} of $S$ is then defined as
\[
\rho(S) =\sup \{\rho(n) \mid n \in S \}.
\]
The elasticity of a semigroup element measures the ``spread'' of its factorization lengths.  One of the advantages of defining the elasticity of an element $n$ as the quotient of the maximum and minimum lengths (as opposed to, say, their difference) is that one cannot obtain larger elasticity values ``for free'' by simply taking multiples of $n$.  Indeed, if $\ell(n) = 3$ and $L(n) = 5$, then $2n$ has factorizations of length $6$ and $10$ obtained by concatenating factorizations of $n$ of minimum and maximum length, respectively.  The only way for $\rho(2n)$ to exceed $\rho(n)$ is for $2n$ to have ``new'' factorizations not obtained from those of $n$.  

We introduce two more terms before exploring an in-depth example.  When the supremum in this expression is attained (i.e., there exists $n \in S$ with $\rho(n) = \rho(S)$) we call the elasticity of $S$ \emph{accepted}\index{elasticity!accepted}.   We say that $S$ is \emph{fully elastic}\index{elasticity!fully~elastic} if for every rational $q \in \mathbb{Q} \cap [1,\rho(S))$ (or $[1,\infty)$ if the elasticity is infinite) there exists a nonzero $n \in S$ such that $\rho(n) = q$.

\begin{example}\label{e:elasticity}
We return to the basic semigroup $S = \<2,3\>$ in Example~\ref{e:2gen} to offer some examples of the calculations thus far suggested.
Hence, each factorization of $n \in S$ has the form
\[
n = 2x_1 + 3x_2.
\] 
Table~\ref{tb:2gen} suggests that factorizations of a given element of $S$ are far from unique in general.  Notice that in $S$, the longest factorization of an element $n \in S$ contains the most possible copies of $2$ and the shortest the most possible copies of $3$.  This is the intuition behind Theorem~\ref{t:maxminlen}:\ for large semigroup elements, a maximum length factorization for $n + 2$ can be obtained a maximum length factorization for $n$ by adding a single copy of $2$.  

Using this fact and some elementary number theory, explicit formulas for all the invariants discussed to this point can be worked out for arbitrary elements of $S$.  For~instance, for all $n \in S$ we have that
\[
\ell(n) = \left\lceil \frac{n}{3} \right\rceil \quad \mbox{and} \quad L(n) = \left\lfloor
\frac{n}{2} \right\rfloor,
\] 
and thus
\[
\rho(n) = \left\lfloor \frac{n}{2} \right\rfloor \big/ \left\lceil
\frac{n}{3} \right\rceil.
\]
Using the fact that 2 copies of the generator 3 can be exchanged for 3 copies of the generator 2 in any factorization in $S$, we 
obtain for $n \geq 4$ that
\[
\mathsf L(n) = \bigg\{\!\!
\left\lceil \frac{n}{3} \right\rceil, 
\left\lceil \frac{n}{3} \right\rceil + 1, \ldots,
\left\lfloor \frac{n}{2} \right\rfloor - 1,
\left\lfloor \frac{n}{2} \right\rfloor
\!\!\bigg\}.
\]
Using the notation $[x,y] = \{z \in \ZZ \,\mid\, x \leq z \leq y \}$ for $x \leq y$ integers, we conclude
\[
\mathcal L(S) = \big\{\! \{1\}, \big[ \lceil \tfrac{4}{3}\rceil, \lfloor \tfrac{4}{2} \rfloor \big],
\big[ \lceil \tfrac{5}{3}\rceil, \lfloor \tfrac{5}{2} \rfloor\big],\ldots \!\big\}.
\] 
Moreover, it is easy to verify that $\rho(n) \leq 3/2$ for all $n \in S$
and that $\rho(n) = 3/2$ if and only if $n \equiv 0 \bmod 6$.  Thus $\rho(S) = 3/2$ and
the elasticity is accepted.  
\end{example}

Though a comparable analysis of elasticities for $S = \<7,10,12\>$ is out of reach, we offer in Figure~\ref{fig:elasticity} a graph of the elasticity values for $S = \<7,10,12\>$ to give the reader a feel for how the elasticity behaves for large elements of $S$.

\begin{figure}[t]
\begin{center}
\includegraphics[width=4.5in]{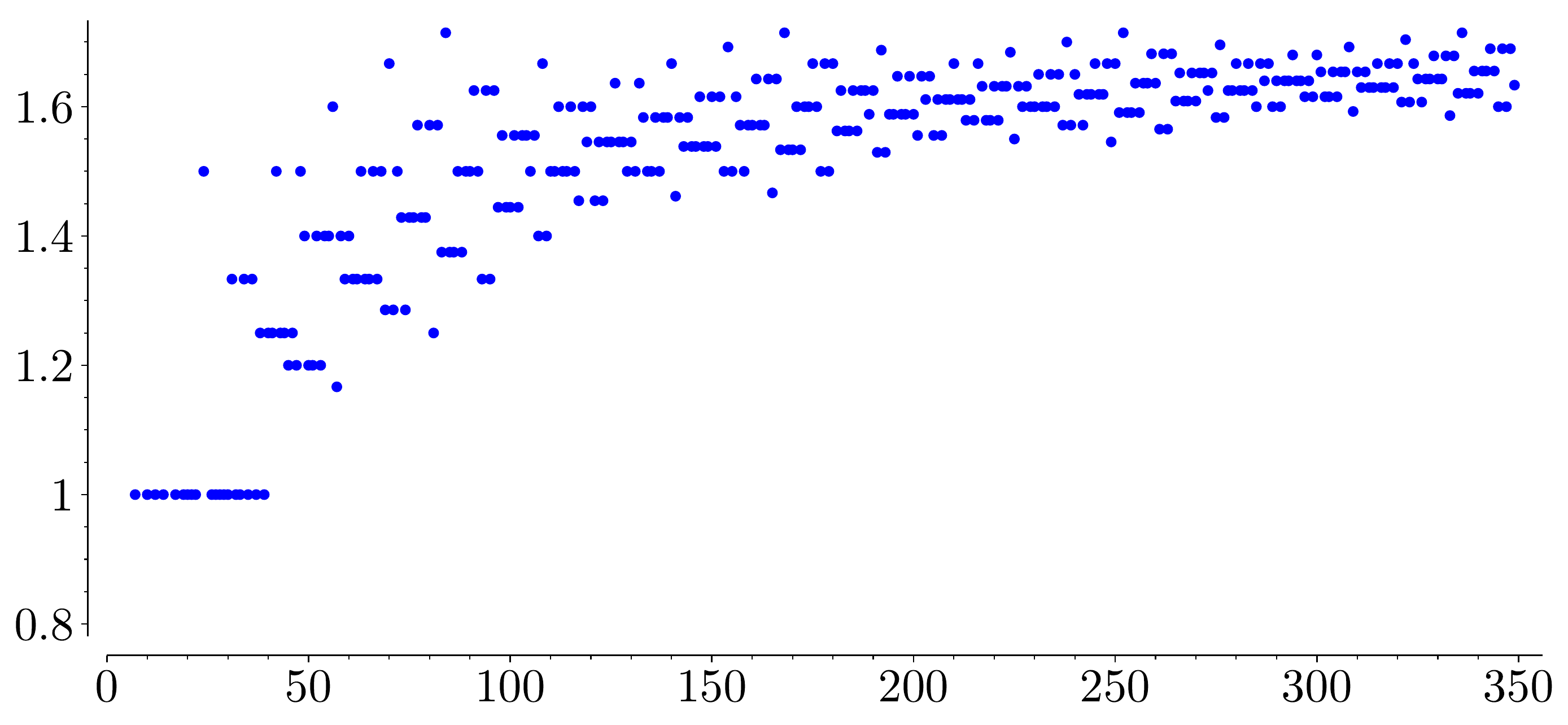}
\end{center}
\caption[Elasticity values for elements of $S = \<7,10,12\>$]{Elasticity values for elements of $S = \<7,10,12\>$.}
\label{fig:elasticity}
\end{figure}

Many basic results concerning elasticity in numerical semigroups are worked out in the
paper~\cite{CHM}, which was a product of a summer REU program.  We review, with proof,
three of that paper's principal results, the first of which yields an exact calculation for the elasticity of $S$.  

\begin{theorem}[{\cite[Theorem 2.1]{CHM}}]\label{t:elasticity}
Let $S = \<n_1, \ldots, n_k\>$ be a numerical semigroup, where $n_1 < n_2 < \cdots < n_k$ is a minimal set of
generators for $S$.  Then $\rho(S) = n_k/n_1$ and the elasticity of $S$ is accepted.
\end{theorem}

\begin{proof}
If $n \in S$ and $n = x_1 n_1 + \ldots + x_kn_k$, then
\[
\frac{n}{n_k} = \frac{n_1}{n_k}x_1 + \ldots + \frac{n_k}{n_k}x_k \le
x_1 + \ldots + x_k \le \frac{n_1}{n_1}x_1 + \ldots + \frac{n_k}{n_1}x_k
 = \frac{n}{n_1}.
\]
Thus $L(n) \le n/n_1$ and $l(n) \ge n/n_k$ for all $n \in S$, so we can conclude $\rho(S) \le n_k/n_1$.  Also, $\rho(S) \ge \rho(n_1 n_k) = n_k/n_1$, yielding equality and acceptance. 
\hfill $\qed$
\end{proof}

We now answer the question of full elasticity in the negative.

\begin{theorem}[{\cite[Theorem 2.2]{CHM}}]\label{t:notfullyelastic}
If $S = \<n_1, \ldots, n_k\>$ be a numerical semigroup, where $2 \le n_1 < \ldots < n_k$ and $k \ge 2$, then $S$ is not fully elastic.  
\end{theorem}

\begin{proof}
Let $N = n_{k-1}n_k + n_1n_k$.  By Theorem~\ref{t:maxminlen}, for each $n > n_{k-1}n_k$, we have
\[
\rho(n + n_1n_k)
= \frac{L(n + n_1n_k)}{\ell(n + n_1n_k)}
= \frac{L(n) + n_k}{\ell(n) + n_1}
\ge \frac{L(n)}{\ell(n)}
= \rho(n)
\]
since $\rho(n) \le n_k/n_1$ by Theorem~\ref{t:elasticity}.  As such, there are only finitely many elements with elasticity less than $\rho(N)$, so $S$ cannot be fully elastic.  \hfill $\qed$
\end{proof}

The proof of Theorem~\ref{t:notfullyelastic} can be used to prove a result
which is of its own interest.  For a numerical semigroup $S$, let
\[
R(S) = \{\rho(n) \,\mid\, n\in S\}.
\] 

\begin{corollary}[{\cite[Corollary 2.3]{CHM}}]\label{c:onelimitpoint}
For any numerical semigroup $S$, the only limit point of $R(S)$ is $\rho(S)$.
\end{corollary}

\begin{proof}
Let $n_1 < n_2 < \cdots < n_k$ be a minimal set of generators for the numerical semigroup $S = \<n_1, \ldots, n_k\>$, where $k \ge 2$.  If $n = a(n_1 n_k) + n_1$ for $a \in \ZZ_{\ge 0}$, then 
\[
\rho(n) = \frac{L(n)}{l(n)} = \frac{an_k + 1}{an_1 + 1}.
\]
It follows that $\rho(n) < n_k/n_1$ for all $a \in \ZZ_{\ge 0}$ and $\lim_{a \to \infty} \rho(n) = n_k/n_1$, making $n_k/n_1$ a limit point of the set $R(S)$.  Additionally, by Theorem~\ref{t:maxminlen}, for $n > n_{k-1}{n_k}$ and $a \ge 1$ we have 
\[
\rho(n + an_1n_k) = \frac{L(n + an_1n_k)}{l(n + an_1n_k)} = \frac{L(n) + an_k}{L(n) + an_1},
\]
meaning $R(S)$ is the union of a finite set (elasticities of the elements less than $n_{k-1}n_k + n_1n_k$) and a union of $n_1n_k$ monotone increasing sequences approaching $n_k/n_1$.  As such, we conclude $n_k/n_1$ is the only limit point.  \hfill $\qed$
%
%
\end{proof}

\noindent
The original proofs in~\cite{CHM} did not use Theorem~\ref{t:maxminlen}, and were much more technical.  The proofs given above are a consequence of a complete description of $R(S)$ in~\cite{elastsets}, a recent paper with an undergraduate co-author in which Theorem~\ref{t:maxminlen} first appeared.  


The elasticity does lend us information concerning the structure of the length set, but only limited information.  While it deals with the maximum and minimum length values, it does not explore the finer structure of $\mathsf L(n)$ (or more generally of~$\mathsf Z(n)$).  There are several invariants studied in the theory of non-unique factorizations which yield more refined information - we introduce one such measure here, which is known as the delta set\index{delta~set}.

Let $S = \<n_1, \ldots, n_t\>$ be a numerical semigroup, where $n_1, \ldots, n_k \in \mathbb{N}$ minimally generate $S$ and $k \ge 2$. 
If $\mathsf L(n) = \{\ell_1, \ldots , \ell_t\}$ with the $\ell_i$'s listed in increasing order, then set
\[
\Delta(n) = \{\ell_i - \ell_{i-1} \,\mid\, 2 \leq i\leq t\}\]
and
\[
\Delta(S) = \bigcup_{0 < n \in S} \Delta(n).
\]
Our hypothesis that $t \geq 2$ ensure $\Delta(S) \neq \emptyset$, since (for instance) if $n = n_1n_2$, then both $n_1$ and $n_2 \in \mathsf L(n)$.  Also, for each $n \in S$, since $|\mathsf L(n)| < \infty$ by Exercise~\ref{exe:mingens}, we clearly have $|\Delta(n)| < \infty$ as well.  

\begin{example}\label{e:deltaset}
We use the calculations already presented in Example~\ref{e:elasticity}.  For $S = \<2,3\>$, our formula for $\mathsf L(n)$ yields for all $n \in S$ that
\[
\Delta(n) = \{1\} \quad \mbox{and thus} \quad \Delta(S)=\{1\}.
\]
Calculations for $S = \<7,10,12\>$ in Example~\ref{e:3gen} require advanced techniques.  From Table~\ref{tb:3gen} we have that $\Delta(24) = \{1\}$ while $\Delta(42)=\{2\}$.  Thus
\[
\Delta(S) \supseteq \{1,2\}.
\]
Using \cite[Corollary 2.3]{BCKR}, we obtain $\max\Delta(S) \le 2$, which yields equality.  
\end{example}

Many basic results concerning the structure of the delta set of a numerical semigroup can be found in \cite{BCKR} (another paper that is the product of an REU project).  The publication of \cite{BCKR} led to a long series of papers devoted to the study of delta sets and related properties in numerical semigroups, which approach delta sets from both theoretical and computational standpoints.  In our bibliography, we offer a subset of this list of papers that include undergraduate co-authors (\cite{ACHP, CDHK, CGLMS, CHK, CKLNZ, COn18, CK}). 

Before proceeding, we will need two fundamental results.  While we state these results in terms of numerical semigroups, they are actually valid for any 
affine semigroup\index{affine~semigroup} (i.e., a subset $S \subset \ZZ_{\ge 0}^d$ closed under vector addition and finitely generated).  We omit the proofs, but invite interested readers to construct proofs specifically for the numerical semigroup setting.  The first merely establishes the finiteness of $\Delta(S)$; proofs can be found in both \cite[Proposition 2.3]{BCKR} or \cite[Theorem 2.5]{CGLMS}.  Another proof can be constructed using our still to come Theorem \ref{t:deltaperiodic}. 

\begin{proposition}\label{p:finitedelta}
If $S$ is a numerical semigroup, then $|\Delta(S)| <\infty$.
\end{proposition}

The second result is a deeper structure theorem concerning delta sets, due to Geroldinger, and a general proof can be found in \cite[Lemma~3]{ag3}.  

\begin{proposition}\label{p:deltagcd}
If $S$ is a numerical semigroup, then
\[
\min\,\Delta(S) = \gcd\,\Delta(S).
\]
Hence, if $d = \gcd\,\Delta(S)$, then
\[
\Delta(S) \subseteq \{d, 2d, \ldots, ad\}
\]
for some $a \in \ZZ_{\ge 0}$.  
\end{proposition}

Proposition~\ref{p:deltagcd} raises two interesting questions, both of which were addressed by the authors of \cite{BCKR}. 
\begin{itemize}
\item 
Given positive integers $d$ and $k$, can one construct a numerical semigroup $S$ with $\Delta(S) = \{d, 2d, \ldots, kd\}$?  

\item 
Must the set containment in Proposition~\ref{p:deltagcd} be an equality?
\end{itemize}
Prior to \cite{BCKR}, all examples in the literature of delta sets (albeit in different settings - primarily in Krull domains and monoids) consisted of a set of consecutive multiples of a fixed positive integer $d$.  For numerical semigroups, on the other hand, the answer to the first question is yes, but the answer to the second is~no.

\begin{proposition}[{\cite[Corollary 4.8]{BCKR}}]\label{p:deltafull}
For each $n \geq 3$ and $k \geq 1$ with $\gcd(n, k) = 1$, the numerical semigroup
\[
S = \<n, n + k, (k+1)n - k\>,
\]
is minimally 3-generated and
\[
\Delta(S) = \bigg\{k, 2k, \ldots, \left\lfloor\frac{n+k-1}{k+2}\right\rfloor k\bigg\}.
\]
Hence, for any positive integers $k$ and $t$, there exists a three generated
numerical semigroup $S$ such that $\Delta(S) = \{k, 2k, \ldots, tk\}$.
\end{proposition}

\begin{proposition}[{\cite[Proposition 4.9]{BCKR}}]\label{p:deltaskips}
For each $n \ge 3$, the numerical semigroup
\[
S = \<n, n + 1, n^2 - n - 1\>,
\]
is minimally 3-generated and
\[
\Delta(S) = \{1, \ldots, n - 2\} \cup \{2n - 5\}.
\]
\end{proposition}

The semigroups in Propositions~\ref{p:deltafull} and~\ref{p:deltaskips} have fairly intuitive minimal generators.  For instance, in Proposition~\ref{p:deltaskips}, $n^2-n-1$ is the Frobenius number of $\<n,n+1\>$, and in Proposition~\ref{p:deltafull}, $S$ reduces to $\<n, n + 1, 2n - 1\>$ when $k = 1$.  Note also that Proposition~\ref{p:deltaskips} gives a ``loud'' no to the second question, as it shows that one can construct as large a ``gap'' as desired in the set $\{d, 2d, \ldots, kd\}$.

 While some fairly deep results have been obtained, a good grasp on the general form for the delta set remains out of reach.  We will do such a computation for arithmetical numerical semigroups (i.e., when $S = \<a, a + d, \ldots, a + kd\>$ for $0 \leq k < a$ and~$\gcd(a, d) = 1$).   This result was originally proved in \cite[Theorem 3.9]{BCKR}, but we present a much shorter self contained proof which later appeared in~\cite{ACHP}, another product of an REU.  We begin with a lemma.

\begin{lemma}[{\cite[Lemma 2.1]{ACHP}}]\label{l:membershiparithmetical}
Let $S$, $a$, $d$, and $k$ be defined as above.
If $n \in S$, then $n = c_1a + c_2d$ with $c_1, c_2 \in \ZZ_{\ge 0}$ and $0 \leq c_2 < a$.
\end{lemma}

\begin{proof}
Any $n \in S$ can be written in the form $c_1a + c_2d$ for some $c_1, c_2 \in \ZZ_{\ge 0}$. Write $c_2 = qa + r$ with $0 \leq r < a$.  Now $n = c_1a + c_2d = a(c_1 + qr) + rd$. \hfill $\qed$
\end{proof}

\begin{theorem}[{\cite[Theorem 2.2]{ACHP}}]\label{t:lengthsetarithmetical}
If $S$, $a$, $d$, and $k$ are defined the same as above, $n = c_1a + c_2d \in S$ with $0 \leq c_2 < a$, and 
\[
K = \frac{c_2 - c_1k}{a + kd},
\]
then we have
\[\mathsf L(n) = \{c_1 + jd \mid K \le j \le 0\}.\]
\end{theorem}

\begin{proof}
Suppose $l \in \mathsf L(n)$.  Now $la \equiv n \equiv c_1a \pmod{d}$, and thus $\mathsf L(n) \subset c_1 + d\ZZ$.  
Writing $l = c_1 + jd$ for $j \in \ZZ$, we see
\[
a(c_1 + jd) = al \le n \le (a + kd)l = (a + kd)(c_1 + jd),
\]
so
\[
K = \frac{c_2 - c_1k}{(a + kd)} = \frac{n - c_1(a + kd)}{(a + kd)d} \le j \le \frac{n - c_1a}{ad} = \frac{c_2}{a} < 1.
\]
This means $\mathsf L(n) \subset \{c_1 + jd \mid K \le j \le 0\}$.  

It remains to locate a factorization of length $c_1 + jd$ for each $j \in \ZZ$ with $K \le j \le 0$.  Write $c_2 - j = qk + r$ for $q, r \in \ZZ$ with $0 \le r < k$.  
We have
\[
\begin{array}{r@{}c@{}l}
n
&{}={}& a(c_1 + jd) + d(c_2 - j)
= a(q + 1 + c_1 + jd - 1 - q) + d(qk + r) \\
&{}={}& q(a + kd) + (a + rd) + (c_1 + jd - 1 - q)a,
\end{array}
\]
which is a factorization of $n$ of length $c_1 + jd$.  Thus $c_1 + jd \in \mathsf L(n)$, as desired. \hfill $\qed$
\end{proof}

An obvious corollary to this theorem follows.

\begin{corollary}\label{c:deltas}
If $S$, $a$, $d$, and $k$ are as defined above, then $\Delta(S) = \{d\}$.  Moreover,
\begin{itemize}
\item if $n_2 > n_1 > 1$ are relatively prime integers, then $\Delta(\<n_1, n_2\>) = \{n_2 - n_1\}$;
\item if $n > 1$ and $k$ are integers with $1 \leq k \leq n - 1$, then for $S = \<n, n + 1, \ldots, n + k\>$ we have
that $\Delta(S) = \{1\}$.
\end{itemize}
\end{corollary}

Additionally, Theorem~\ref{t:lengthsetarithmetical} can be used to show that $\mathcal L(S)$ is not a perfect invariant, that is, one cannot in general recover a given numerical semigroup $S$ from $\mathcal L(S)$.  

\begin{cproblem}\label{prob:lensetcharacterize}
Use Theorem~\ref{t:lengthsetarithmetical} to find two numerical semigroups $S_1$ and $S_2$ so
that $\mathcal{L}(S_1)=\mathcal{L}(S_2)$ but $S_1\neq S_2$.
\end{cproblem}

We close this section with another REU related result which appears in \cite{CHK}.  Writing the elements of a numerical semigroup $S$ in order as $s_1, s_2, \ldots$, where $s_i < s_{i+1}$ for all $i \geq 1$, we now consider the sequence of sets
\[
\Delta(s_1), \Delta(s_2), \Delta(s_3), \ldots
\]
In the case where $S$ is arithmetical, then for large $i$ this sequence is comprised solely of $\{k\}$, which is not too interesting.  Using Table~\ref{tb:3gen}, one can construct the beginning of this sequence for $S = \<7,10,12\>$:
\[
\emptyset,\emptyset,\ldots,\emptyset, \{1\}, \emptyset,\emptyset,\emptyset,\emptyset,\emptyset, \{1\}, \emptyset, \emptyset, \{1\}, \emptyset, \{1\}, \emptyset, \{1\}, \emptyset, \{1\}, \{1\}, \{2\}, \{1\}, \ldots 
\]
While the beginning behavior of these sequences is in some sense ``chaotic,'' in the long run, they are much more well behaved.  This can be better demonstrated with some graphs.  Figure~\ref{fig:7-10-12-delta} represents the sequence of delta sets for $S = \<7,10,12\>$, while Figure~\ref{fig:6-9-20-delta} does so for the Chicken McNugget semigroup.  On these graphs, a point is plotted at $(n,d)$ if $d \in \Delta(n)$.  
 
\begin{figure}[t]
\begin{center}
\includegraphics[width=3.0in]{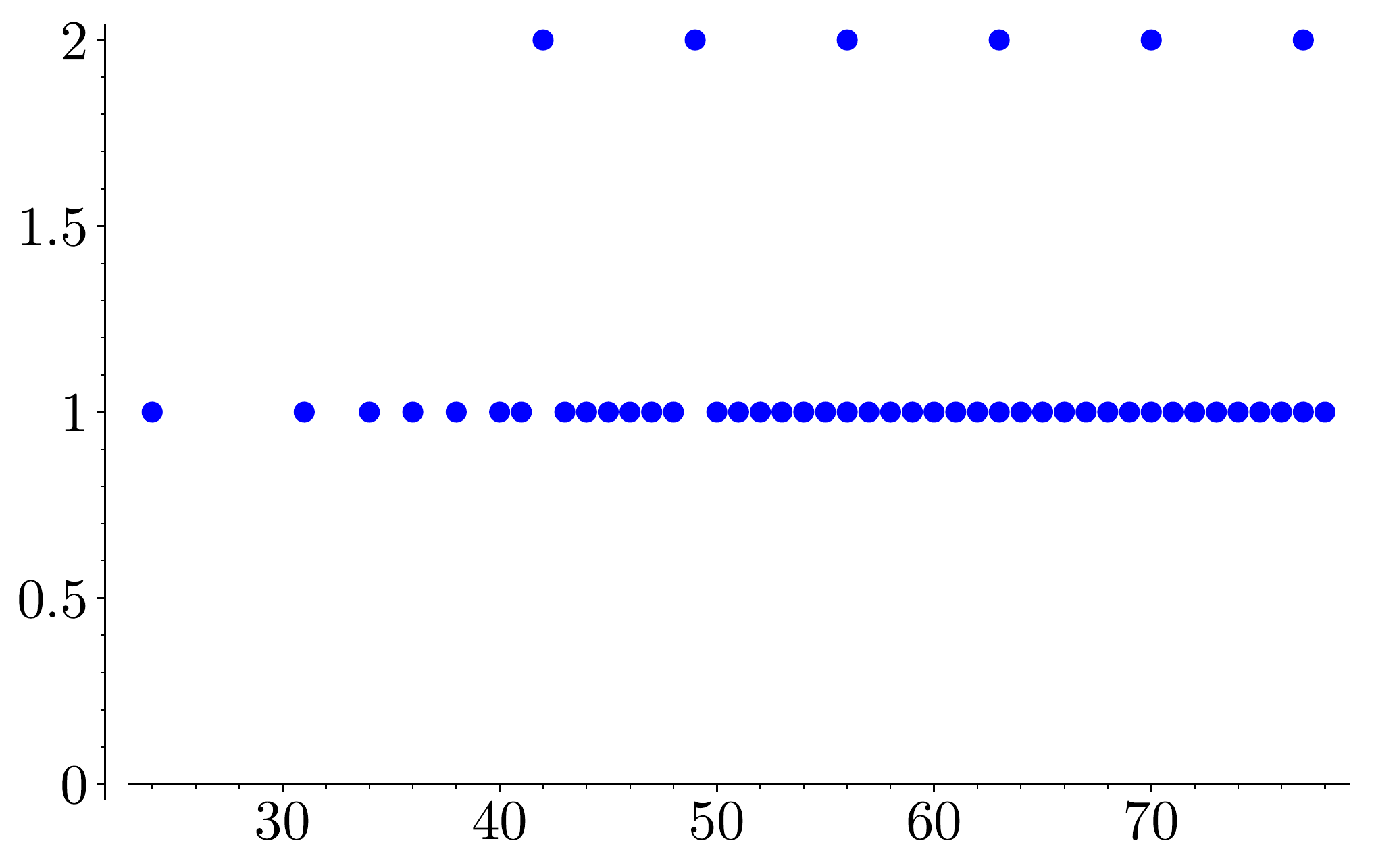}
\end{center}
\caption[Delta set values for $S = \<7,10,12\>$]{Delta set values for $S = \<7,10,12\>$.}
\label{fig:7-10-12-delta}
\end{figure}
 
\begin{figure}[t]
\begin{center}
\includegraphics[width=4.5in]{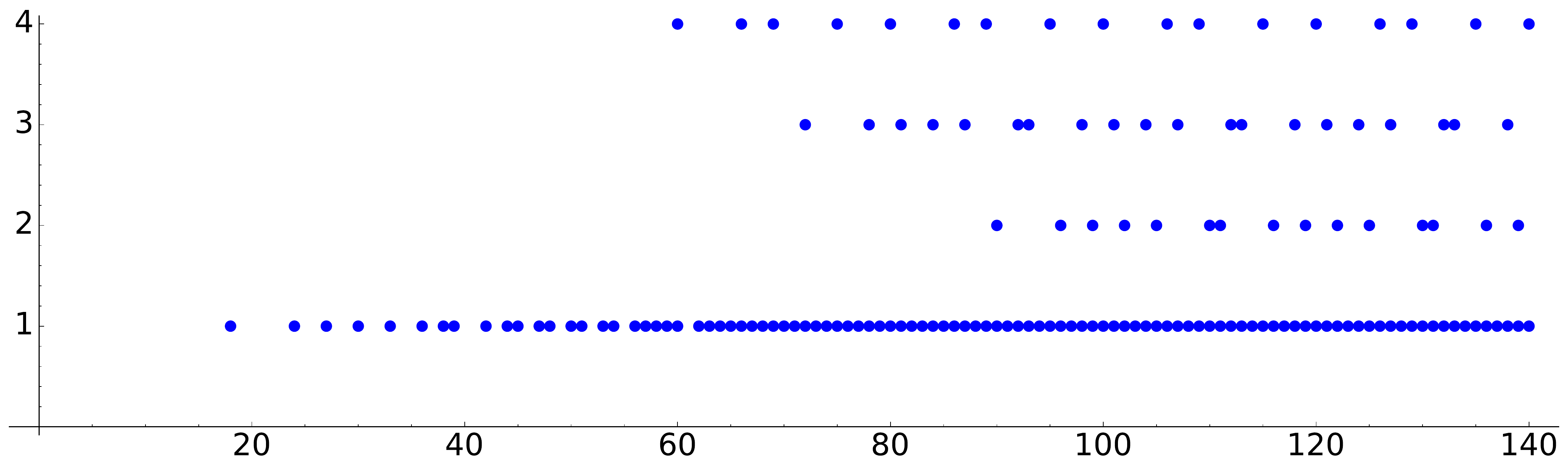}
\end{center}
\caption[Delta set values for $S = \<6,9,20\>$]{Delta set values for $S = \<6,9,20\>$.}
\label{fig:6-9-20-delta}
\end{figure}

Using data such as the above, it was conjectured shortly after the publication of~\cite{BCKR} that this sequence of sets is eventually periodic.  Three years later, this problem was solved, again as part of an REU project.  

\begin{theorem}[{\cite[Theorem 1]{CHK}}]\label{t:deltaperiodic}
Given a numerical semigroup $S = \<n_1, \ldots, n_k\>$ with $n_1 < n_2 < \cdots < n_k$ and $N = 2kn_2n_k^2 + n_1n_k$, we have 
\[
\Delta(n) = \Delta(n - n_1n_k)
\]
for every $n \ge N$.  Hence, 
\[
\Delta(S) = \bigcup_{n \in S, \, n < N} \Delta(n)
\]
\end{theorem}
 
The importance of Theorem~\ref{t:deltaperiodic} cannot be overstated, as it turns the problem of computing $\Delta(S)$ into a finite time exercise.  The bound $N$ given in Theorem~\ref{t:deltaperiodic} has been drastically improved in \cite{GGMV} (Table 1 in that paper shows exactly how drastic this improvement is).  An alternate view of the computation of $\Delta(S)$ using the \emph{Betti numbers} of $S$ can be found in~\cite{CGLMS}, which is also an REU product.


\section{Using software to guide mathematical inquisition}
\label{sec:computing}

One of the most reliable tools when working with numerical semigroups is computer software.  We will give an overview of using the \texttt{GAP} package \texttt{numericalsgps}.   \texttt{GAP} (Groups, Algorithms, Programming) is a computer algebra system used in a variety of discrete mathemetical areas, and \texttt{numericalsgps} is a package for working specifically with numerical semigroups, including over 400 pre-programmed functions to compute numerous invariants and properties of numerical semigroups.  Full documentation for this and other packages can be found on the \texttt{GAP} website.  
\begin{center}
\url{https://www.gap-system.org/}
\end{center}
We begin by providing a brief overview of the functionality related to the topics covered in the previous section.  Once \texttt{GAP} is up and running, you must first load the \texttt{numericalsgps} package.  

{\small
\begin{verbatim}
gap> LoadPackage("numericalsgps");
true
\end{verbatim}
}

\noindent
Once this is done, you can begin to compute information about the numerical semigroups you are interested in examining, such as the Frobenius number.  

{\small
\begin{verbatim}
gap> McN:= NumericalSemigroup(6,9,20);
<Numerical semigroup with 3 generators>
gap> FrobeniusNumberOfNumericalSemigroup(McN);
43
\end{verbatim}
}

Many of the quantities the \texttt{numericalsgps} package can compute center around factorizations and their lengths.  Given how central the functions that compute $\mathsf Z(n)$ and $\mathsf L(n)$ are, these functions have undergone numerous improvements since the early days of the \texttt{numericalsgps} package, and now run surprisingly fast even for reasonably large input.  

{\small
\begin{verbatim}
gap> FactorizationsElementWRTNumericalSemigroup(50, McN);
[ [ 5, 0, 1 ], [ 2, 2, 1 ] ]
gap> FactorizationsElementWRTNumericalSemigroup(60, McN);
[ [ 10, 0, 0 ], [ 7, 2, 0 ], [ 4, 4, 0 ], 
  [ 1, 6, 0 ], [ 0, 0, 3 ] ]
gap> LengthsOfFactorizationsElementWRTNumericalSemigroup(60, McN);
[ 3, 7, 8, 9, 10 ]
gap> LengthsOfFactorizationsElementWRTNumericalSemigroup(150, McN);
[ 10, 11, 13, 14, 15, 16, 17, 18, 19, 20, 21, 22, 23, 24, 25 ]
\end{verbatim}
}

\noindent
The \texttt{numericalsgps} package can also compute delta sets, both of numerical semigroups and of their elements.  The original implementation of the latter function used Theorem~\ref{t:deltaperiodic} to compute the delta set of every element up to $N$, and only more recently was a more direct algorithm developed~\cite{affineinvariantcomp}.  

{\small
\begin{verbatim}
gap> DeltaSetOfFactorizationsElementWRTNumericalSemigroup(60, McN);
[ 1, 4 ]
gap> DeltaSetOfNumericalSemigroup(McN);
[ 1, 2, 3, 4 ]
\end{verbatim}
}

One of the primary goals is to use these observations to formulate meaningful conjectures among these concepts, and even to aide in the development of a proof.  In what follows, we hope to give a better sense of this process by walking through a specific example.  

Suppose we decide to study maximum factorization length.  We begin by using the \texttt{numericalsgps} package to compute maximum factorization lengths for elements of $S = \<7, 10, 12\>$ from Example~\ref{e:3gen}.  

{\small
\begin{verbatim}
gap> S := NumericalSemigroup(7,10,12);
<Numerical semigroup with 3 generators>
gap> FactorizationsElementWRTNumericalSemigroup(60,S);
[ [ 0, 6, 0 ], [ 4, 2, 1 ], [ 2, 1, 3 ], [ 0, 0, 5 ] ]
gap> Maximum(
> LengthsOfFactorizationsElementWRTNumericalSemigroup(60,S));
7
\end{verbatim}
}

Using built-in \texttt{GAP} functions, maximum factorization length can be computed for several semigroup elements in one go.  We first compute a list of the initial elements of~$S$ (this avoids an error message when attemting to compute $\mathsf Z(n)$ when $n \notin S$).  

{\small
\begin{verbatim}
gap> elements := Filtered([1..60], n -> (n in S));
[ 7, 10, 12, 14, 17, 19, 20, 21, 22, 24, 26, 27, 28, 29, 30, 31, 
 32, 33, 34, 35, 36, 37, 38, 39, 40, 41, 42, 43, 44, 45, 46, 47, 
 48, 49, 50, 51, 52, 53, 54, 55, 56, 57, 58, 59, 60 ]
\end{verbatim}
}

Next, we compute the values of $L(n)$ for semigroup elements $n \le 100$.  

{\small
\begin{verbatim}
gap> List(elements, n -> Maximum(
> LengthsOfFactorizationsElementWRTNumericalSemigroup(n,S)));
[ 1, 1, 1, 2, 2, 2, 2, 3, 2, 3, 3, 3, 4, 3, 3, 4, 3, 4, 4, 5, 
  4, 4, 5, 4, 5, 5, 6, 5, 5, 6, 5, 6, 6, 7, 6, 6, 7, 6, 7, 7, 
  8, 7, 7, 8, 7 ]
\end{verbatim}
}

\begin{figure}[t]
\begin{center}
\includegraphics[width=4.5in]{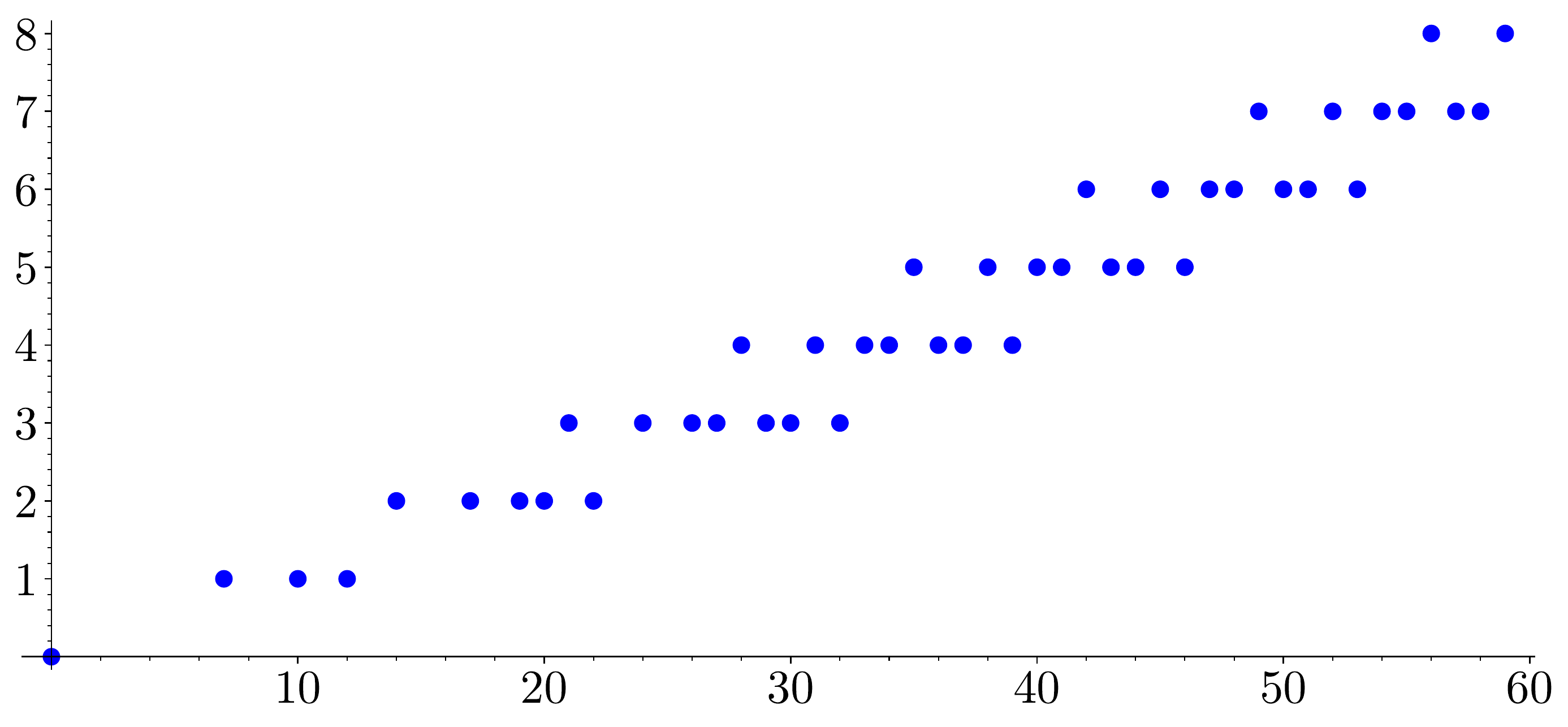}
\end{center}
\caption[Maximum factorization length values for elements of $S = \<7,10,12\>$]{Maximum factorization length values for elements of $S = \<7,10,12\>$.}
\label{fig:maxlen}
\end{figure}

Well-chosen plots can be an incredibly effective tool for visualizing such data.  Figure~\ref{fig:maxlen} depicts the values output above; the repeating pattern in the right half of the plot is undeniable.  This is what we call a \emph{quasilinear} function, that is, a linear function with periodic coefficients.  For our particular $S$ and $n \ge 26$, we have
\[
L(n) = \tfrac{1}{7}n + a(n)
\]
where $a(n)$ is a periodic function with period 7 (for instance, $a(n) = -\tfrac{11}{7}$ whenever $n \equiv 4 \bmod 7$, so that $L(60) = \tfrac{1}{7}(60) - \tfrac{11}{7} = 7$).  The resulting plot resembles 7 parallel lines, each with slope $\tfrac{1}{7}$.  Another way to express a quasilinear function with constant linear coefficient is via constant successive differences, i.e.\ 
\[
L(n + 7) - L(n) = 1
\]
for all $n \ge 26$.  Notice that this is precisely the relation claimed in Theorem~\ref{t:maxminlen}.  

In addition to elucidating the quasilinear pattern, computations can also be used help us work towards a proof.  We begin by asking \texttt{numericalsgps} to compute the factorizations of maximum length for $n = 44$, $51$, and $58$ (each $7$ apart).  

{\small
\begin{verbatim}
gap> LengthsOfFactorizationsElementWRTNumericalSemigroup(44,S);
[ 4, 5 ]
gap> LengthsOfFactorizationsElementWRTNumericalSemigroup(51,S);
[ 5, 6 ]
gap> LengthsOfFactorizationsElementWRTNumericalSemigroup(58,S);
[ 5, 6, 7 ]
gap> Filtered(FactorizationsElementWRTNumericalSemigroup(44,S), 
>             f -> (Sum(f)=5));
[ [ 2, 3, 0 ] ]
gap> Filtered(FactorizationsElementWRTNumericalSemigroup(51,S), 
>             f -> (Sum(f)=6));
[ [ 3, 3, 0 ] ]
gap> Filtered(FactorizationsElementWRTNumericalSemigroup(58,S),                  
>             f -> (Sum(f)=7));
[ [ 4, 3, 0 ] ]

\end{verbatim}
}

Notice the only change in the factorizations is the first coordinate, which increases by exactly $1$ each time the element $n$ increases by exactly $7$.  Intuitively, this is because longer factorizations should use more small generators.  This identifies where the period of $7$ and the leading coefficient of $\tfrac{1}{7}$ originate.  However, this does not yet explain why the quasilinear pattern does not begin until $n = 26$.  After testing our conjecture on several more numerical semigroups, we come across an example that provides some insight behind this final piece of the puzzle.  

{\small
\begin{verbatim}
gap> S2 := NumericalSemigroup(9,10,21);
<Numerical semigroup with 3 generators>
gap> FactorizationsElementWRTNumericalSemigroup(41,S2);
[ [ 0, 2, 1 ] ]
gap> FactorizationsElementWRTNumericalSemigroup(50,S2);
[ [ 0, 5, 0 ], [ 1, 2, 1 ] ]
gap> FactorizationsElementWRTNumericalSemigroup(59,S2);
[ [ 1, 5, 0 ], [ 2, 2, 1 ] ]
\end{verbatim}
}

Here, we see the longest factorization of $n = 50$ in $S_2 = \<9, 10, 21\>$ does not use any copies of the smallest generator.  As it turns out, this phenomenon can only happen for small semigroup elements, as once $n$ is large enough, any factorization with no copies of the smallest generator can be ``traded'' for a longer factorization that does.  This highlights the key to proving Theorem~\ref{t:maxminlen}:\ determining how large $n$ must be to ensure all of its factorizations of maximal length have at least one copy of the smallest generator.  

We invite the reader to use the ideas discussed above to obtain a rigorous proof of Theorem~\ref{t:maxminlen} (indeed, the proof appearing in~\cite{elastsets} utilizes these ideas).  

\begin{cproblem}\label{cp:maxminlen}
Prove Theorem~\ref{t:maxminlen}.  
\end{cproblem}

\section{Research projects:\ asymptotics of factorizations}
\label{sec:asymptotics}

The length of a factorization coincides with the $\ell_1$-norm of the corresponding point.  Much like Theorem~\ref{t:maxminlen}, several other norms that arise in discrete optimization appear to have EQP behavior.  The following was observed during the 2017 San Diego State University Mathematics REU, and motivates the research project that follows.  

In what follows, for $\mathbf a \in \ZZ^k$ and $r \in \ZZ_{\ge 1}$, let
\[
\|\mathbf a\|_r = (a_1^r + \cdots + a_k^r)^{1/r}
\]
and
\[
\|\mathbf a\|_\infty = \max(a_1, \ldots, a_k),
\]
which are known as the $\ell_r$- and $\ell_\infty$-norm, respectively.  

\begin{cproblem}
Let $S = \<n_1, \ldots, n_k\>$.  Prove the function
$$\ell_\infty(n) = \min\{\|\mathbf a\|_\infty : \mathbf a \in \mathsf Z_S(n)\}$$
is eventually quasilinear with period $n_1 + \cdots + n_k$.  
\end{cproblem}

\resproject{
Determine for which fixed $r \in [2, \infty)$ the functions
$$\mathsf M_r(n) = \max\{(\|\mathbf a\|_r)^r : \mathbf a \in \mathsf Z_S(n)\}$$
and
$$\mathsf m_r(n) = \min\{(\|\mathbf a\|_r)^r : \mathbf a \in \mathsf Z_S(n)\}$$
are eventually quasipolynomial.  
}

Given a numerical semigroup $S = \<n_1, \ldots, n_k\>$ with $n_1 \le \cdots \le n_k$, one can define
$$N(\ell) = |\{n \in S : \ell \in S\}|,$$
which counts the number of elements of $S$ with a given length $\ell$ in their length~set.  Unlike many functions discussed above, which take semigroup elements as input, $N$~takes factorization lengths as input.  Since each semigroup element counted by $N(\ell)$ must lie between $n_1 \ell$ and $n_k \ell$, we see that 
$$N(\ell) \le (n_k - n_1)\ell,$$
so $N(\ell)$ grows at most linearly in $\ell$.  This yields the following natural question.  

\resproject{
Fix a numerical semigroup $S$.  Determine whether
$$N(\ell) = |\{n \in S : \ell \in S\}|$$
is eventually quasilinear in $\ell \ge 0$.  
}

One of the running themes of results in the numerical semigroups literature is that the factorization structure is ``chaotic'' for small elements, but ``stabilizes'' for large elements.  Typically, the latter is easier to describe, as evidenced by the word ``eventually'' in several of the results presented above.  Broadly speaking, it would be interesting to determine how much of a numerical semigroup's structure can be recovered from that of its ``large'' elements.  The following project is an initial step in this direction, and at its heart is the question ``does the eventual behavior of a given factorization invariant uniquely determine its behavior for the whole semigroup?''  

\resproject{
Given a numerical semigroup $S = \<n_1, \ldots, n_k\>$ satisfying $n_1 < \cdots < n_k$, Theorem~\ref{t:maxminlen} implies
$$\mathsf M_S(n) = \textstyle\frac{1}{n_1}n + a_S(n) \qquad \textnormal{and} \qquad \mathsf m_S(n) = \textstyle\frac{1}{n_k}n + b_S(n)$$
for some periodic functions $a_S(n)$ and $b_S(n)$.  Characterize the functions $a_S(n)$ and $b_S(n)$ in terms of the generators of $S$.  For distinct numerical semigroups $S$ and $T$, is it possible that $a_S(n) = a_T(n)$ or $b_S(n) = b_T(n)$ for all $n$?  
}

Most of the invariants introduced thus far (and indeed, most in the literature) are derived from ``extremal'' factorizations.  In a recent REU project, ``medium'' factorization lengths were studied.  More precisely, the \emph{length multiset}
$$\mathsf L_S[\![n]\!] = \{\!\{|\mathbf a| : \mathbf a \in \mathsf Z_S(n)\}\!\}$$
was defined, wherein factorization lengths are considered with repetition, and the following quantities were considered:
\begin{itemize}
\item 
$\mu_S(n)$, the \emph{mean} of the elements of $\mathsf L_S[\![n]\!]$; and

\item 
$\eta_S(n)$, the \emph{median} of the elements of $\mathsf L_S[\![n]\!]$.  

\end{itemize}
Notice that $\mathsf L_S[\![n]\!]$ has the same cardinality as $\mathsf Z_S(n)$, since factorization lengths are counted with repetition in $\mathsf L_S[\![n]\!]$.  

\begin{example}\label{e:lengthmultiset}
The length multiset of $n = 1400$ in $S = \<5, 7, 8\>$ is depicted in Figure~\ref{fig:lengthmultiset}, wherein a point at $(\ell, m)$ indicates the length $\ell$ appears exactly $m$ times in $\mathsf L[\![n]\!]$.  The~lengths in $\mathsf L[\![n]\!]$ range from $175$ to $280$ (as~predicted by Theorem~\ref{t:elasticity}), and the mode length(s) occur around $n/7$ (note that $7$ is the middle generator of $S$).  In this case, as $n \to \infty$, the histogram approaches a \emph{triangular distribution}.  The second histogram in Figure~\ref{fig:lengthmultiset} is for the length multiset of an element of $S = \<5, 8, 9, 11\>$ and has a visually different shape.  Indeed, the limiting distribution of the length multiset is only triangular for 3-generated numerical semigroups.  
\end{example}

\begin{figure}[bt]
\includegraphics[width=2.25in]{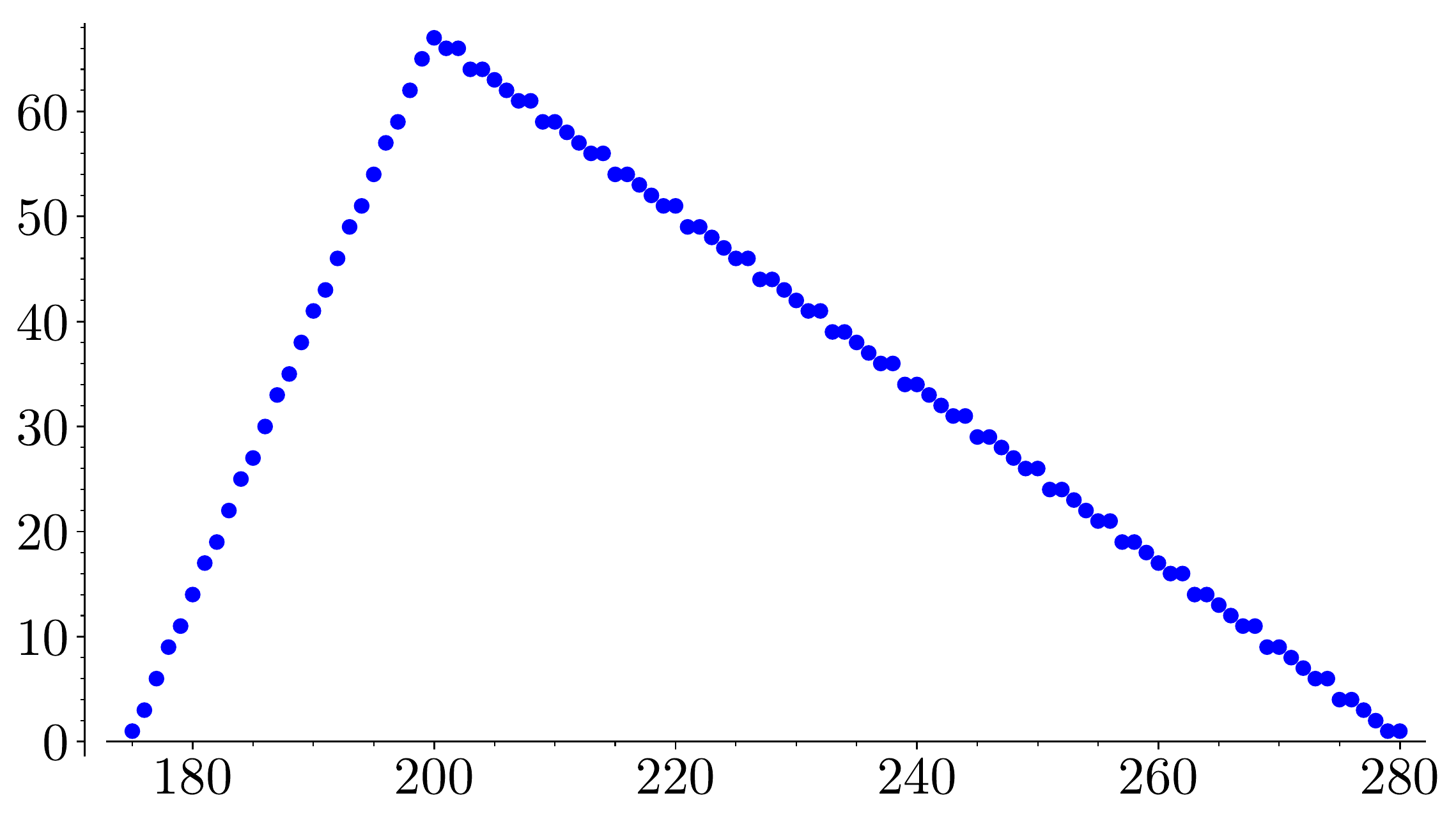}
\hspace{0.1in}
\includegraphics[width=2.25in]{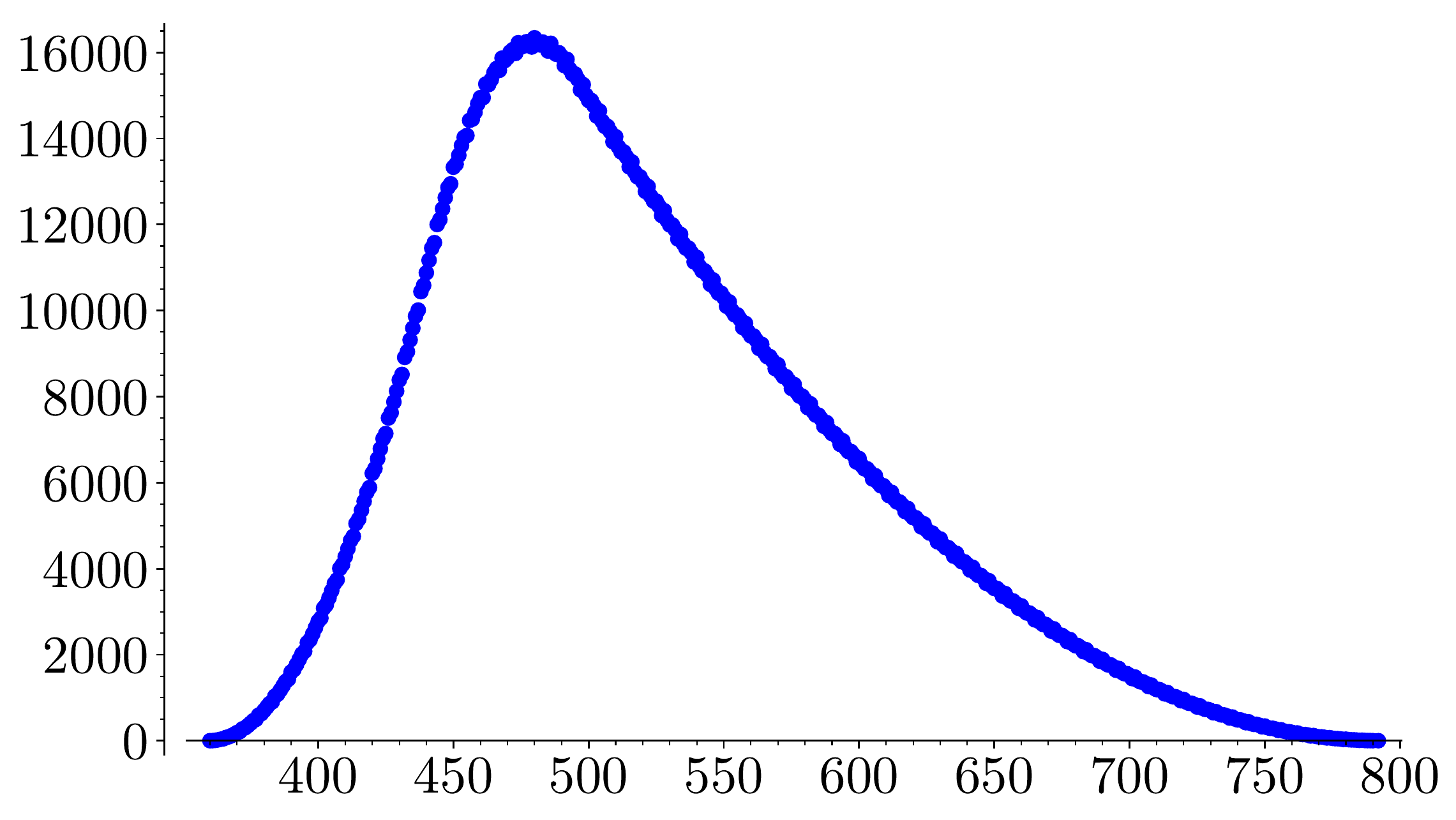}
%
%
\caption{Histograms of the length multiset $\mathsf L_S[\![1400]\!]$ for $S = \<5, 7, 8\>$ (left) and the length multiset $\mathsf L_S[\![3960]\!]$ for $S = \<5, 8, 9, 11\>$ (right).}
\label{fig:lengthmultiset}       
\end{figure}

The following result will appear in a forthcoming paper, and implies that although $\mu(n)$ is not itself (eventually) quasipolynomial, it can be expressed in terms of quasipolynomial functions.  

\begin{theorem}\label{t:mean}
Fix a numerical semigroup $S = \<n_1, \ldots, n_k\>$.  The function $\mu(n)$ equals the quotient of two quasipolynomial functions, and 
$$\lim_{n \to \infty} \frac{\mu_S(n)}{n} = \frac{1}{k}\left(\frac{1}{n_1} + \cdots + \frac{1}{n_k}\right).$$
\end{theorem}

Median factorization length has proven more difficult to describe in general.  The~limiting distribution of $\mathsf L[\![n]\!]$ is characterized for 3-generated numerical semigroups in~\cite{lengthdist1}, yielding the following theorem regarding the asymptotic growth rate of median factorization length in this case.  

\begin{theorem}\label{t:median}
Fix a numerical semigroup $S = \<n_1, n_2, n_3\>$, and let 
$$F = \frac{n_1(n_3 - n_2)}{n_2(n_3 - n_1)}$$
(called the \emph{fulcrum constant}).  We have
$$\lim_{n \to \infty} \frac{\eta_S(n)}{n} = \left\{
\begin{array}{l@{\quad}l}
\displaystyle \frac{1}{n_1} \bigg( 1 - \sqrt{ \frac{1 - F}{2} } \bigg) + \frac{1}{n_3} \sqrt{ \frac{1 - F}{2} }
& \textnormal{if } F \le \frac{1}{2},
\\[1.0em]
\displaystyle \frac{1}{n_1} \sqrt{ \frac{F}{2} } + \frac{1}{n_3} \bigg( 1 - \sqrt{ \frac{F}{2} } \bigg)
& \textnormal{if } F \ge \frac{1}{2},
\end{array}\right.$$
the value of which is irrational for some, but not all, numerical semigroups.  
\end{theorem}

Theorem~\ref{t:median} is a stark contrast to many of the invariants discussed above, since if the limit therein is irrational, then $\eta(n)$ cannot possibly coincide with a quasipolynomial for large $n$ (indeed, this follows from the fact that any linear function sending at least 2 rational inputs to rational outputs must have rational coefficients).  As~such, studying the asymptotic behavior of median factorization length requires different techniques than previously studied invariants.  

As the histograms in Figure~\ref{fig:lengthmultiset} demonstrate, the limiting distribution of the length multiset for 3-generated numerical semigroup elements differs drastically from semigroups with more generators.  Students interested in the following project are encouraged to begin by reading~\cite{lengthdist1}, wherein the limiting distribution is carefully worked out in the 3-generated case.

\resproject{
Fix a numerical semigroup $S$.  Find a formula for
$$\lim_{n \to \infty} \frac{\eta(n)}{n},$$
the asymptotic growth rate of the median factorization length of $n$.  
}

\section{Research projects:\ random numerical semigroups}
\label{sec:random}

Suppose someone walks up to you on the street and hands you a ``random'' numerical semigroup.  What do we expect it to look like?  Is it more likely to have a lot of minimal generators, or only a few?  How large do we expect its Frobenius number to be?  How many gaps do we expect it to have?  Such questions of ``average'' or ``expected'' behavior arise frequently in discrete mathematics, and often utilize tools from probability and real analysis that are otherwise uncommon in discrete settings.  

The general strategy is to define a \emph{random model} that selects a mathematical object ``at random'', and then determine the probability that the chosen object has a particular property.  One prototypical example comes from graph theory:\ given a fixed integer $n$ and probability $p$, select a random graph $G$ on $n$ vertices by deciding, with independent probability $p$, whether to draw an edge between each pair of vertices $v_1$ and $v_2$.  A natural question to ask is ``what is the probability $G$ is connected?'' (note that the larger $p$ is, the more edges one expects to draw, and thus the higher chance the resulting graph is connected).  This is a difficult question to obtain an exact answer for, although estimates can be obtained for small $n$ (with the help of computer software) since there are only finitely many graphs with $n$ vertices.  That said, it turns out that for very large~$n$, there is an $\epsilon > 0$ so that
\begin{itemize}
\item 
if $p < \log(n)/n - \epsilon$, then $G$ has low probability of being connected, and 

\item 
if $p > \log(n)/n + \epsilon$, then $G$ has high probability of being connected, 

\end{itemize}
where $\epsilon \to 0$ as $n \to \infty$.  Here, the phrases ``low probability'' and ``high probability'' mean probability tending to $0$ and $1$, respectively, as $n \to \infty$.  This kind of bifurcation (illustrated in Figure~\ref{fig:ergraphs} for varying values of $n$) is a phenomenon known as a \emph{threshold function}, and occurs frequently when answering probabilistic questions in discrete mathematics.  

\begin{figure}[t]
\includegraphics[width=4.5in]{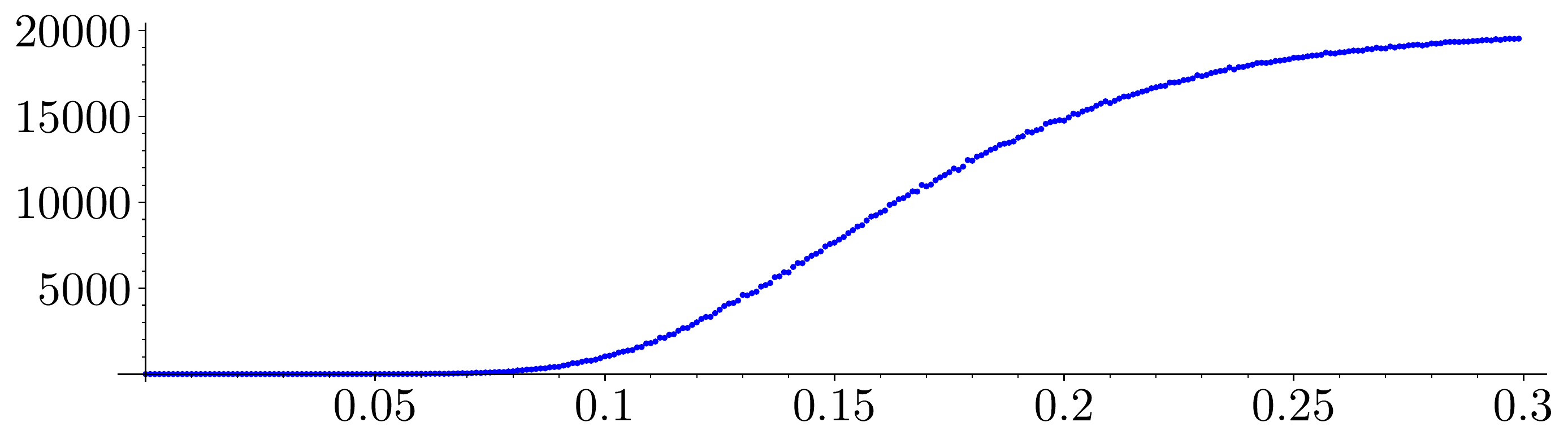}
\\
\includegraphics[width=4.5in]{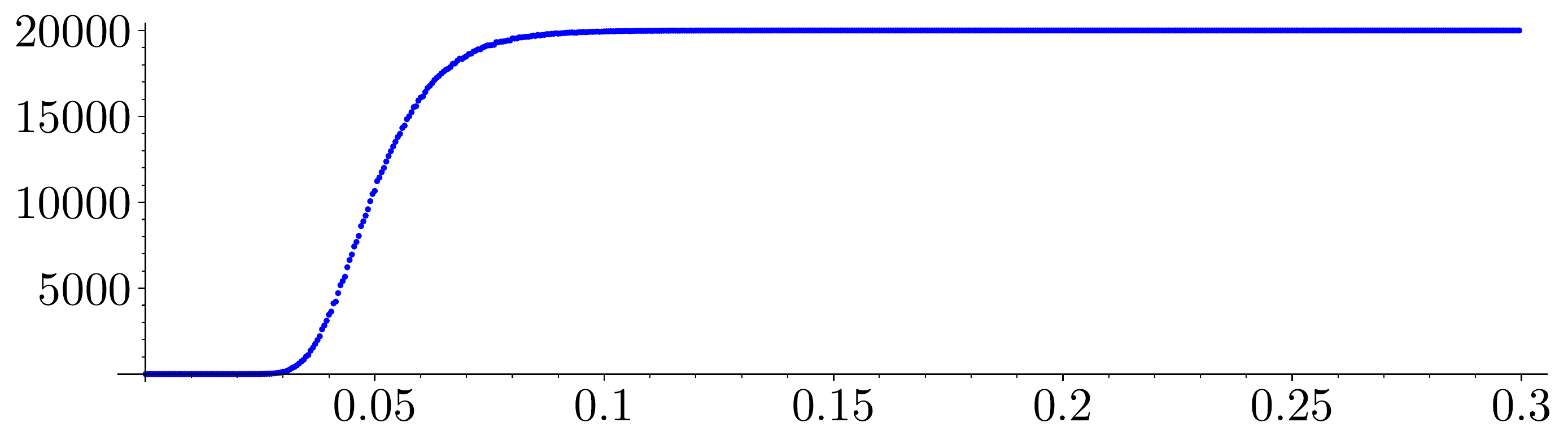}
%
%
\caption{Scatterplots recording the number of connected graphs sampled as a function of $p$ for $n = 20$ and $n = 100$, respectively.  Sample size is 20,000 for each plotted value of $p$.}
\label{fig:ergraphs}
\end{figure}

The authors of~\cite{rnscomplex} introduce a model of selecting a numerical semigroup at random that is similar to the above model for random graphs.  Their model takes two inputs $M \in \ZZ_{\ge 1}$ and $p \in [0, 1]$, and randomly selects a numerical semigroup by selecting a generating set $A$ that includes each integer $n = 1, 2, \ldots, M$ with independent probability $p$.  For example, if $M = 40$ and $p = 0.1$, then one possible set is $A = \{6,9,18,20,32\}$ (this~is not unreasonable, as one would expect $4$ to be selected on average).  However, only 3 elements of $A$ are minimal generators, since $18 = 9 + 9$ and $32 = 20 + 6 + 6$.  As such, the selected semigroup $S = \langle A \rangle = \langle 6, 9, 20 \rangle$ has embedding dimension~3.  

Regarding the expected properties of numerical semigroups selected in this way, two main results are proven in~\cite{rnscomplex}.  First, the threshold function for whether or not the resulting numerical semigroup $S$ has finite complement is proven to be~$1/M$.  More precisely, for each large $M$, there exists $\epsilon > 0$ (with $\epsilon \to 0$ as $M \to \infty$) so
\begin{itemize}
\item 
if~$p < 1/M - \epsilon$, then $\ZZ_{\ge 0} \setminus S$ is finite with low probability, and

\item 
if $p > 1/M + \epsilon$, then $\ZZ_{\ge 0} \setminus S$ is finite with high probability.

\end{itemize}
One expects $|A| = 1$ on average when $p = 1/M$, meaning for large $M$, the selected numerical semigroup is most likely to either have finite complement (if $p > 1/M$) or equal the trivial semigroup~$\{0\}$ (if $p < 1/M$).  

The remaining results in~\cite{rnscomplex} provide lower and upper bounds on the expected number of minimal generators of the selected semigroup $S$.  More precisely, a formula is obtained for $\mathbb E [e(S)]$ in terms of $p$ and $M$ (though it is computationally infeasible for large $M$), and derives from it lower and upper bounds on $\lim_{M \to \infty} \mathbb E [\mathsf e(S)]$ for fixed $p$.  It is also shown that 
$$\lim_{M \to \infty} \mathbb E [\mathsf e(S)] = \frac{p}{1-p} \lim_{M \to \infty} \mathbb E [\mathsf g(S)],$$
thereby providing lower and upper bounds on the expected number of gaps as well.

Asymptotic estimates of this nature can be useful, for instance, in testing conjectures in semigroup theory.  Suppose a researcher has a conjecture regarding numerical semigroups with exactly 150 gaps.  They could test their conjecture on a small number of ``larger'' numerical semigroups selected using a random model, choosing the parameters so as to maximize the chances of selecting a numerical semigroup with 150 gaps.  

This random model is just one of many possible models for randomly selecting numerical semigroups, and using different models is likely to yield different expected behavior for the resulting semigroups.  Given below are some alternative models that have yet to be explored.  
The first adds a new parameter to the existing model, namely the multiplicity of the semigroup, yielding more control over which semigroups are selected.  The second selects \emph{oversemigroups} (that is, semigroups containing a given semigroup) instead of generators, and takes their intersection.  


\resproject{
Study random numerical semigroups selected using the following model:\ given $M, m \in \ZZ_{\ge 1}$ and $p \in [0,1]$, select the semigroup
$$S = \<\{m\} \cup A\> \cup ([M+1, \infty) \cap \ZZ)$$
by selecting a random subset $A \subset [m+1, M] \cap \ZZ$ that includes each integer  the original model discussed above.  
}

\resproject{
Study random numerical semigroups selected using the following model:\ given $N \in \ZZ_{\ge 1}$ and $p \in [0,1]$, select the semigroup 
$$S = \bigcap_{2 \le a < b \le N} \<a, b\>$$
where each numerical semigroup $\<a, b\>$ with $\gcd(a,b) = 1$ is included in the intersection with independent probability $p$.  
}

For each of these projects, a natural starting place would be to use computer software to produce a large sample of numerical semigroups and compute the average number of minimal generators, Frobenius number, etc.\ as estimates of their expected value for varying choices of the parameters.







\begin{acknowledgement}
The authors would like to thank Nathan Kaplan for discussions related to this work.
All plots created using Sage~\cite{sage}, and all computations involving numerical semigroups utilize the \texttt{GAP} package \texttt{numericalsgps} \cite{numericalsgps}.
\end{acknowledgement}

\printindex
\end{document}